\documentclass[12pt]{article}
\usepackage{amsmath, amsthm, amscd, amsfonts, amssymb, graphicx}
\usepackage{titles}
\usepackage{algorithm}
\usepackage{algpseudocode}
\usepackage{subfigure} 
\usepackage{caption}
\usepackage{graphics,graphicx}
\usepackage{amsmath, amssymb, amsthm}
\usepackage{mathrsfs} 
\usepackage{enumitem}

\usepackage{indentfirst}

\numberwithin{equation}{section}
\usepackage{epsfig}
\usepackage{color}
\usepackage{fancyhdr}
\usepackage{titlesec}
\usepackage{tikz}
\usepackage{afterpage}
\usetikzlibrary{arrows}
\usetikzlibrary{decorations.markings}
\tikzstyle{block}=[draw opacity=0.7,line width=1.4cm]
\tikzset{
	big black arrow/.style={
		decoration={markings,mark=at position 1 with {\arrow[scale=2.5,black]{>}}},
		postaction={decorate},
		shorten >=0.4pt},
	line/.style={draw, ->}}

\usepackage{eufrak}
\setbox0=\hbox{$+$}
\newdimen\plusheight
\plusheight=\ht0
\def\+{\;\lower\plusheight\hbox{$+$}\;}

\setbox0=\hbox{$-$}
\newdimen\minusheight
\minusheight=\ht0
\def\-{\;\lower\minusheight\hbox{$-$}\;}

\setbox0=\hbox{$\cdots$}
\newdimen\cdotsheight
\cdotsheight=\plusheight
\def\cds{\lower\cdotsheight\hbox{$\cdots$}}
\newtheorem{theorem}{Theorem}[section]
\newtheorem{lemma}[theorem]{Lemma}
\newtheorem{corollary}[theorem]{Corollary}

\theoremstyle{definition}

\theoremstyle{remark}
\newtheorem{remark}[theorem]{Remark}
\numberwithin{equation}{section}

\allowdisplaybreaks
\title{\textbf{Proofs of the Conjectures on $SOME(n)$ and $DSOME(n)$ Functions Related to Integer Partitions}}
\date{}

\begin{document}
\maketitle
\vspace{-2cm}
\begin{center}
	{\bf Gaurab Bardhan$^1$ and Nipen Saikia$^{2, \ast}$}\\
	$^1$Department of Mathematics, Tyagbir Hem Baruah College,\\ Jamugurihat, Sonitpur, Assam, India.\\
	E. Mail: gaurabbardhan561@gmail.com
	\vskip2mm
	$^2$Department of Mathematics, Rajiv Gandhi
	University,\\ Rono Hills, Doimukh, Arunachal Pradesh, India.\\
	E. Mail(s): nipennak@yahoo.com\\
	$^\ast$\textit{Corresponding author}.\end{center}
\begin{abstract}
Andrews and  Dastidar  (2026) introduced   $SOME(n)$ and 
$DSOME(n)$ functions related to  partitions  of a positive integer $n$, where $SOME(n)$ is the sum of all the odd parts in the partitions of $n$ minus the sum of all the even parts and $DSOME(n)$ is the
sum of all the odd parts in the partitions of $n$ into distinct parts minus
the sum of all the even parts in the same partitions. The purpose of this paper is to establish the conjecture
$SOME(\lambda)\equiv0\pmod{5^\alpha}$, $\alpha\ge 1$ and $\lambda\ge0$ are integers such that
$24\lambda\equiv1\pmod{5^\alpha}$ due to Andrews and  Dastidar, and the conjecture  $DSOME(50n+21)\equiv0\pmod{8}$ due to Baruah and Gogoi (2026). In the process, we establish  some new infinite families of congruences modulo 2, 4, and 8 for  $DSOME(n)$.
\end{abstract}


\section{Introduction}
For a positive integer $n$ and complex numbers $a$ and $q$ with $|q|<1$, define the standard 
$q$-Pochhammer symbols by 
$$(a;q)_0=1,\qquad
(a;q)_n=\prod_{j=0}^{n-1}(1-aq^j),\qquad \mbox{and}\qquad
(a;q)_\infty=\prod_{j=0}^{\infty}(1-aq^j).$$
For brevity,   we write
$$f_t:=(q^t;q^t)_\infty$$ for every positive integer $t$.

A partition of a positive integer $n$ is a finite non-increasing sequence of positive integers
$\beta_1\geq\beta_2\geq\cdots\geq\beta_k>0$ satisfying
$n=\sum_{i=1}^{k}\beta_i$. The integers $\beta_i$ are called the parts of the partition.
If $p(n)$ denotes the number of partitions of $n$, then the  generating function of $p(n)$ due to Euler  \cite{euler1748introductio} is given by
\begin{equation}\label{SDI3}
\sum_{n=0}^{\infty}p(n)q^n
=\dfrac{1}{(q;q)_\infty}
=\dfrac{1}{f_1},
\qquad p(0)=1.
\end{equation}
Also, if  $p_{d}(n)$ denotes the number of partitions of $n$ wherein parts are distinct then its generating function is given by 
\begin{equation}\label{DSM8M25}
\sum_{n=0}^{\infty}p_{d}(n)q^n=\dfrac{(q^2;q^2)_\infty}{(q;q)_\infty}=\dfrac{f_2}{f_1}.   
\end{equation}
Andrews and Dastidar~\cite{MG} introduced $SOME(n)$  and $DSOME(n)$ functions related to the  partitions  of a positive integer $n$, where $SOME(n)$ is the sum of all the odd parts in the partitions of $n$ minus the sum of all the even parts and $DSOME(n)$ is the
sum of all the odd parts in the partitions of $n$ into distinct parts minus
the sum of all the even parts. They obtained the  generating functions of $SOME(n)$  and $DSOME(n)$  as
\begin{equation}\label{SDI4}
\sum_{n=0}^{\infty}SOME(n)q^n
=
\dfrac{1}{(q;q)_\infty}
\sum_{m=1}^{\infty}\dfrac{q^m}{(1+q^m)^2}
\end{equation}
and
\begin{equation}\label{bg1}\sum_{n=0}^{\infty}DSOME(n)q^n
=
(-q;q)_\infty
\sum_{m=1}^{\infty}
\dfrac{(-1)^{m-1}q^m}{(1+q^m)^2}.\end{equation}
and established following Ramanujan-type congruences for $SOME(n)$  and $DSOME(n)$ functions
\begin{align*}
SOME(4n)&\equiv0\pmod{4},\\
DSOME(4n)&\equiv0\pmod{4},\\
SOME(5n+2)&\equiv0\pmod{5},\\
\intertext{and}
SOME(5n+4)&\equiv0\pmod{5}.\\
\end{align*}
They also gave the following conjecture \cite[p. 8, (3)]{MG}: For integers  $\alpha\ge 1$  and $\lambda \ge0$  whenever $24\lambda\equiv1\pmod{5^\alpha}$, we have
\begin{equation}\label{sc}
SOME(\lambda)\equiv0\pmod{5^\alpha}.
\end{equation}
Baruah and Gogoi~\cite{NDB}  also proved some congruences for $DSOME(n)$ modulo $4$ and $8$ by expressing \eqref{bg1} in the form
\begin{equation}\label{SDI10}
	\sum_{n=0}^{\infty}DSOME(n)q^n
	=
	\dfrac{1}{8}
	\left(
	\dfrac{f_2}{f_1}
	-
	\dfrac{f_1^7}{f_2^3}
	\right).
\end{equation} They also proposed the following conjecture modulo 8 \cite[p. 14, Conjecture 4.1]{NDB}:
\begin{equation}\label{SDI11}
DSOME(50n+21)\equiv0\pmod{8}.
\end{equation}

The main objective of this paper is to give  proofs of conjectures \eqref{sc} and \eqref{SDI11}.
The layout of the paper is as follows:  In Sect. 2 we record some preliminary results for ready reference. In Sect 3, we first prove the conjecture \eqref{sc} related to $SOME(n)$ function. In Sect. 4,  we prove some new congruences modulo 2, 4, and 8 for $DSOME(n)$. and establish the conjecture  \eqref{SDI11}. 

\section{Preliminary results}
Ramanujan's general theta function $f(\alpha, \beta)$ \cite[p. 34, (18.1)]{BBC} is defined by
$$
   f(\alpha, \beta) = \sum_{m=-\infty}^{\infty} \alpha^{m(m+1)/2} \beta^{m(m-1)/2}, \qquad |\alpha\beta| < 1.
$$
Three important special cases for $f(\alpha, \beta)$ are the functions $\phi(q), \psi(q)$, and $f(-q)$ \cite[p. 35, Entry 18]{BBC}, which are defined as
\begin{equation}\label{phi}
    \phi(q) := f(q, q) = \sum_{m=-\infty}^{\infty} q^{m^2} = \dfrac{f_2^5}{f_1^2f_4^2},
\end{equation}
    $$\psi(q) := f(q, q^3) = \sum_{m=0}^{\infty} q^{m(m+1)/2} = \dfrac{f_2^2}{f_1},$$
and 
\begin{equation}\label{DSE19}
    f(-q) := f(-q, -q^2) = \sum_{m=-\infty}^{\infty} (-1)^m q^{m(3m-1)/2} = f_1.
\end{equation}
Also, from \eqref{phi}, we have 
\begin{equation}\label{DSE3}
\phi(-q)=\dfrac{f_1^2}{f_2}
=1+2X(q),
\end{equation}
where 
\begin{equation}\label{xq}
X(q)=\sum_{m=1}^{\infty}(-1)^m q^{m^2}.
\end{equation}

In the following lemmas, let  $\sigma(n)$ denote the sum of divisors of a positive integer  $n$, where $\sigma(x)=0$ if $x$ is not an integer.
\begin{lemma}We have
    \begin{equation}\label{DSMN3}
        q\dfrac{\dfrac{d}{dq} f_1}{f_1}=
-\sum_{n=1}^{\infty}\sigma(n)q^n.
    \end{equation}
\end{lemma}
\begin{proof}
By logarithmic differentiation, we obtain
\begin{align}
 q\dfrac{\dfrac{d}{dq} f_1}{f_1}
&=
q\dfrac{d}{dq}
\left(
\sum_{m=1}^{\infty}\log(1-q^m)
\right)
\notag\\
&\label{SMNE2}=
-\sum_{m=1}^{\infty}
\dfrac{mq^m}{1-q^m}
\\
&=
-\sum_{m=1}^{\infty}
\sum_{j=1}^{\infty}m q^{mj}
\notag\\
&=
-\sum_{n=1}^{\infty}\sigma(n)q^n.\notag
\end{align}
\end{proof}
\begin{lemma} We have
\begin{equation}\label{DSE13}
    \sum_{m=1}^{\infty}
\dfrac{q^m}{(1+q^m)^2}\equiv\sum_{n=1}^{\infty}\sigma(n)q^n\pmod4.
\end{equation}
\end{lemma}
\begin{proof}We have
\begin{align}
\sum_{m=1}^{\infty}
\dfrac{q^m}{(1+q^m)^2}
&=
\sum_{m=1}^{\infty}
\sum_{j=1}^{\infty}
(-1)^{j-1}j q^{mj}
\notag\\
&=
\sum_{n=1}^{\infty}
\left(
\sum_{j\mid n}(-1)^{j-1}j
\right)q^n
\notag\\
&=
\sum_{n=1}^{\infty}
\left(
\sum_{\substack{j\mid n\\j\ {\rm odd}}}j
-
\sum_{\substack{j\mid n\\j\ {\rm even}}}j
\right)q^n
\notag\\
&=
\sum_{n=1}^{\infty}
\left(
\sigma(n)
-
2\sum_{\substack{j\mid n\\j\ {\rm even}}}j
\right)q^n
\notag\\
&\label{SMN1}=
\sum_{n=1}^{\infty}
\left(
\sigma(n)-4\sigma(n/2)
\right)q^n\\
&\equiv
\sum_{n=1}^{\infty}\sigma(n)q^n
\pmod4.\notag
\end{align}
\end{proof}
\begin{lemma}
    For $x\in\mathbb{C}$, with $|x|<1$, we have
    \begin{equation}\label{DSM8M17}
\dfrac{x(1+x^2)}{(1+x)^4}
=
\sum_{j=1}^{\infty}
(-1)^{j-1}
\dfrac{j(j^2+2)}{3}x^j.
    \end{equation}
\end{lemma}
\begin{proof}
By the generalized binomial theorem, we note that
\begin{equation}\label{DSM8M14}
\dfrac{1}{(1+x)^4}
=
\sum_{j=0}^{\infty}
(-1)^j
\binom{j+3}{3}x^j.
\end{equation}
Employing \eqref{DSM8M14} and using
$\dbinom{j}{3}=0$ for $j=1,2$, we obtain
\begin{align}
\dfrac{x(1+x^2)}{(1+x)^4}
&=
x\sum_{j=0}^{\infty}
(-1)^j\binom{j+3}{3}x^j
+
x^3\sum_{j=0}^{\infty}
(-1)^j\binom{j+3}{3}x^j
\notag\\
&=
\sum_{j=1}^{\infty}
(-1)^{j-1}\binom{j+2}{3}x^j
+
\sum_{j=3}^{\infty}
(-1)^{j-3}\binom{j}{3}x^j
\notag\\
&\label{DSM8M15}=
\sum_{j=1}^{\infty}
(-1)^{j-1}
\left(
\binom{j+2}{3}+\binom{j}{3}
\right)x^j.
\end{align}
Also, 
\begin{equation}\label{DSM8M16}
\binom{j+2}{3}+\binom{j}{3}
=
\dfrac{j(j+1)(j+2)}{6}
+
\dfrac{j(j-1)(j-2)}{6}
=
\dfrac{j(j^2+2)}{3}.
\end{equation}
Employing \eqref{DSM8M16} in \eqref{DSM8M15}, we complete the proof.
\end{proof}
\begin{lemma}\label{lem24}
   For all positive integers $s$ and $n$, we have 
   \begin{equation}\label{DSM8M20}
       \sum_{j\mid n}(-1)^{j-1}j^s=
\sigma_s(n)-2^{s+1}\sigma_s(n/2),
   \end{equation}
   where $\sigma_s(n)=\sum_{j\mid n}j^s$ and $\sigma_s(n/2)=0$ whenever $n$ is odd.
\end{lemma}
\begin{proof} Let $s$ be any positive integer.
Considering  the odd and even
divisors of $n$ separately, we obtain
\begin{align}
\sum_{j\mid n}(-1)^{j-1}j^s
&=
\sum_{\substack{j\mid n\\j\ {\rm odd}}}j^s
-
\sum_{\substack{j\mid n\\j\ {\rm even}}}j^s
\notag\\
&=
\left(
\sum_{j\mid n}j^s
-
\sum_{\substack{j\mid n\\j\ {\rm even}}}j^s
\right)
-
\sum_{\substack{j\mid n\\j\ {\rm even}}}j^s
\notag\\
&=
\sigma_s(n)
-
2\sum_{\substack{j\mid n\\j\ {\rm even}}}j^s
\notag\\
&=
\sigma_s(n)
-
2\sum_{d\mid n/2}(2d)^s
\notag\\
&=
\sigma_s(n)-2^{s+1}\sigma_s(n/2).\notag
\end{align}
\end{proof}

\section{Proof of the conjecture \eqref{sc} of $SOME(n)$}
This section is devoted to establishing the conjecture \eqref{sc} due to Andrews and Dastidar.

\begin{theorem}
Let $\alpha\geq1$ and $\lambda\geq0$ be integers such that
$24\lambda\equiv1\pmod{5^\alpha}$.
Then
$$SOME(\lambda)\equiv0\pmod{5^\alpha}.$$
\end{theorem}

\begin{proof}Define $c(n)$, as in \cite[(1.28)]{WY}, by
\begin{equation}\label{SOMEC10}
\sum_{n=0}^{\infty}c(n)q^n
=
\dfrac{2E_2(q^2)-E_2(q)}{f_2}
=
\dfrac{\mathcal{M}(q)}{f_2},
\end{equation}
where $E_2(q)$ \cite[(4.1.7)]{bc} is the classical quasimodular Eisenstein series defined by
\begin{equation}\label{SMNE1}
    E_2(q)=1-24\sum_{k=1}^{\infty}\dfrac{kq^k}{1-q^k},
\end{equation}
and
\begin{equation}\label{SOMEC4}
\mathcal{M}(q):=\sum_{n=0}^{\infty}m(n)q^n=2E_2(q^2)-E_2(q).
\end{equation}
Employing \eqref{DSMN3} and  \eqref{SMNE2}  in \eqref{SMNE1}, we obtain
\begin{equation}\label{SOMEC5}
E_2(q)=1-24\sum_{n=1}^{\infty}\sigma(n)q^n.
\end{equation}
Wang and Yang \cite[Theorem~1.1]{WY} proved that, for $k\geq1$ and $n\geq0$,
\begin{align}
c\left(5^{2k-1}n+\dfrac{7\cdot5^{2k-1}+1}{12}\right)
&\label{WY1}\equiv0\pmod{5^{2k-1}}, \\\intertext{and}
c\left(5^{2k}n+\dfrac{11\cdot5^{2k}+1}{12}\right)
&\label{WY2}\equiv0\pmod{5^{2k}}. 
\end{align}
For integers $N\ge 1$ and $\alpha\geq1$, suppose that
\begin{equation}\label{WY3}
12N\equiv1\pmod{5^\alpha}.
\end{equation}
Set,
\begin{equation}\label{WY5}
r_\alpha=
\begin{cases}
\dfrac{7\cdot5^\alpha+1}{12}, & \alpha \text{ is odd},\\[8pt]
\dfrac{11\cdot5^\alpha+1}{12}, & \alpha \text{ is even}.
\end{cases}
\end{equation}
Then $r_\alpha\in\mathbb Z$, \,$0<r_\alpha<5^\alpha$, and
$12r_\alpha\equiv1\pmod{5^\alpha}$. Since
$\gcd(12,5^\alpha)=1$, from \eqref{WY3}, we obtain
\begin{equation}\label{WY6}
N\equiv r_\alpha\pmod{5^\alpha}.    
\end{equation}
Then for some $n\geq0$, from  \eqref{WY6}, we obtain
\begin{equation}\label{WYe7}
N=5^\alpha n+r_\alpha=
\begin{cases}
5^{2k-1}n+\dfrac{7\cdot5^{2k-1}+1}{12},
& \alpha=2k-1,\\[8pt]
5^{2k}n+\dfrac{11\cdot5^{2k}+1}{12},
& \alpha=2k.
\end{cases}    
\end{equation}
Employing \eqref{WYe7} in \eqref{WY1} (or \eqref{WY2}) if  $\alpha$ is odd (or even), we obtain
\begin{equation}\label{WY7}
c(N)\equiv0\pmod{5^\alpha},
\end{equation}
whenever $$12N\equiv1\pmod{5^\alpha}.$$ 
Let  $\sigma_{od}(n)$ denote the sum of odd divisors of a positive integer $n$, then
\begin{align}
\sum_{n=1}^{\infty}\sigma_{od}(n)q^n
&=\sum_{\substack{m\geq1\\ m\ {\rm odd}}}
\sum_{r=1}^{\infty}mq^{mr} \notag\\
&=\sum_{m=1}^{\infty}\dfrac{mq^m}{1-q^m}
-\sum_{m=1}^{\infty}\dfrac{2mq^{2m}}{1-q^{2m}} \notag\\
&=\sum_{n=1}^{\infty}\sigma(n)q^n
-2\sum_{n=1}^{\infty}\sigma(n)q^{2n} \notag\\
&\label{SOD1}=\sum_{n=1}^{\infty}
\left(\sigma(n)-2\sigma(n/2)\right)q^n, 
\end{align}
where we used the fact that $\sigma(n/2)=0$ if $n$ is an odd  integer.\\\\
Now, employing \eqref{SOMEC5} and \eqref{SOD1} in \eqref{SOMEC4}, we obtain
\begin{align}
\sum_{n=0}^{\infty}m(n)q^n
&=2\left(1-24\sum_{n=1}^{\infty}\sigma(n)q^{2n}\right)-\left(1-24\sum_{n=1}^{\infty}\sigma(n)q^{n}\right) \notag\\
&=1+24\sum_{n=1}^{\infty}
\left(\sigma(n)-2\sigma(n/2)\right)q^n \notag\\
&\label{SOMEC7}=1+24\sum_{n=1}^{\infty}\sigma_{od}(n)q^n.
\end{align}
Since  $\sigma_{od}(2n)=\sigma_{od}(n)$, \eqref{SOMEC7} implies
\begin{equation}\label{SOMEC9}
m(2n)=m(n).
\end{equation}
Employing \eqref{SDI3} and  \eqref{SOMEC4} in \eqref{SOMEC10}, we obtain
\begin{align}
\sum_{n=0}^{\infty}c(n)q^n
&=
\mathcal{M}(q)\sum_{r=0}^{\infty}p(r)q^{2r} \notag\\
&=\left(\sum_{n=0}^{\infty}m(n)q^n\right)\left(\sum_{r=0}^{\infty}p(r)q^{2r}\right) \notag\\
&=\sum_{n=0}^{\infty}\sum_{r=0}^{\lfloor n/2\rfloor}m(n-2r)p(r)q^n.\notag
\end{align}
Extracting the terms with even powers of $q$, replacing $q^2$ by $q$ and simplifying by  employing \eqref{SOMEC9}, we obtain
\begin{align}
    \sum_{n=0}^{\infty}c(2n)q^n
    &=\sum_{r=0}^{n}m(2n-2r)p(r)q^n\notag\\
    &=\sum_{r=0}^{n}m(n-r)p(r)q^n \notag \\
&=\dfrac{\mathcal{M}(q)}{f_1}.\notag
\end{align}
Therefore, \eqref{SOMEC4} implies
\begin{equation}\label{SOMEC13}
\sum_{n=0}^{\infty}c(2n)q^n
=
\dfrac{2E_2(q^2)-E_2(q)}{f_1}.
\end{equation}
Again, from  \eqref{SDI4}, we note that
\begin{equation}\label{SOMEC14}
\sum_{n=0}^{\infty}SOME(n)q^n
=
\dfrac{1}{f_1}\sum_{j=1}^{\infty}\dfrac{q^j}{(1+q^j)^2}.
\end{equation}
Combining  \eqref{SMN1} and \eqref{SOMEC5}, we obtain
\begin{align}
\sum_{j=1}^{\infty}\dfrac{q^j}{(1+q^j)^2}
&=\sum_{j=1}^{\infty}\sum_{r=1}^{\infty}
(-1)^{r-1}r q^{jr} \notag\\
&=\sum_{n=1}^{\infty}
\left(\sigma(n)-4\sigma(n/2)\right)q^n \notag\\
&\label{SOMEC15}=\dfrac{4E_2(q^2)-E_2(q)-3}{24}.
\end{align}
Also, by logarithmic differentiation, we obtain
\begin{equation}\label{po1}
q\dfrac{d}{dq}\dfrac{1}{f_1}=\dfrac{1}{f_1}\sum_{n=1}^{\infty}\sigma(n)q^n.\end{equation}
Employing \eqref{SOMEC5} in \eqref{po1}, we obtain
\begin{equation}\label{SMN5}
    q\dfrac{d}{dq}\dfrac{1}{f_1}=\dfrac{1-E_2(q)}{24f_1}.
\end{equation}
Employing \eqref{SOMEC13}, \eqref{SOMEC15}, \eqref{po1}, and \eqref{SMN5} in \eqref{SOMEC14}, we obtain
\begin{align}
24\sum_{n=0}^{\infty}SOME(n)q^n
&=\dfrac{1}{f_1}\left(4E_2(q^2)-E_2(q)-3\right) \notag\\
&=2\dfrac{2E_2(q^2)-E_2(q)}{f_1}
+\dfrac{E_2(q)-3}{f_1} \notag\\
&=2\sum_{n=0}^{\infty}c(2n)q^n
-24q\dfrac{d}{dq}\dfrac{1}{f_1}-\dfrac{2}{f_1}\notag\\
&\label{SOMEC17}=2\sum_{n=0}^{\infty}c(2n)q^n-24\sum_{n=0}^{\infty}np(n)q^n-2\sum_{n=0}^{\infty}p(n)q^n.
\end{align}
Comparing the coefficients of $q^n$ on both sides of \eqref{SOMEC17},
we obtain
\begin{equation}\label{SOMEC18}
12SOME(n)=c(2n)-(12n+1)p(n).
\end{equation}
Setting $N=2\lambda$ in \eqref{WY7}, we have
\begin{equation}\label{SOMEC20}
c(2\lambda)\equiv0\pmod{5^\alpha}
\end{equation}whenever $24\lambda\equiv1\pmod{5^\alpha}$.
Also, from \cite{R} (see also \cite{Wat}), for $24\lambda\equiv1\pmod{5^\alpha}$,
\begin{equation}\label{SOMEC22}
p(\lambda)\equiv0\pmod{5^\alpha}.
\end{equation} 
Employing \eqref{SOMEC20} and \eqref{SOMEC22} in \eqref{SOMEC18}, we obtain
\begin{equation}\label{a1}
12SOME(\lambda)\equiv0\pmod{5^\alpha}.\end{equation}
Since $\gcd(12,5^\alpha)=1$, the desired result easily follows from \eqref{a1}.
\end{proof}

\section{Congruences modulo 2, 4 and 8 of $DSOME(n)$ and proof of the conjecture \eqref{SDI11} }
 In this section, we  prove some other congruences modulo 2, 4 and 8 satisfied by the function $DSOME(n)$. We establish  the conjecture \eqref{SDI11} in Theorem \ref{nd1} due to Baruah and Gogoi.  

Let  $A(q):=\sum_{r=0}^{\infty}a(r)q^r$ and $B(q)=\sum_{s=0}^{\infty}b(s)q^s$ be any two formal power series. Let $U_2$ be an operator that extracts terms with even indices from a power series. That is,   $U_2[\sum_{r=0}^{\infty}a(r)q^r]=\sum_{r=0}^{\infty}a(2r)q^r$. Then, 
\begin{equation}\label{DSE16}
U_2\left[A(q^2)B(q)\right]
=
U_2\left[
\sum_{r,s\geq0}a(r)b(s)q^{2r+s}
\right]
=
\sum_{r,j\geq0}a(r)b(2j)q^{r+j}
=
A(q)U_2\left[B(q)\right].
\end{equation}
and 
\begin{equation}\label{DSE17}
U_2\left[q\dfrac{d}{dq}B(q)\right]
=
U_2\left[
\sum_{s=0}^{\infty}s\cdot b(s)q^s
\right]
=
\sum_{s=0}^{\infty}2s\cdot b(2s)q^s
=
2q\dfrac{d}{dq} U_2\left[B(q)\right].
\end{equation}

\begin{theorem}\label{DSM2THM}
For any integer $n\ge0$, we have
$$DSOME(n)\equiv
\begin{cases}
1\pmod2,
& n=24j^2+10j+1 \text{ or \;$24j^2+22j+5$}\text{ for some }j\in\mathbb Z,\\
0\pmod2,
& \text{otherwise}.
\end{cases}$$
\end{theorem}

\begin{proof}
For every partition of $n$ into distinct parts,  the difference
between the sum of its odd parts and the sum of its even parts is congruent
to $n$ modulo $2$. Therefore,
\begin{equation*}
DSOME(n)\equiv n p_d(n)\pmod{2}
\end{equation*}
and by using \eqref{DSM8M25}, we obtain
\begin{align}
\sum_{n=0}^{\infty}DSOME(n)q^n
&\equiv \sum_{n=0}^{\infty}n p_d(n)q^n\pmod2\notag\\
&\equiv q\dfrac{d}{dq}
  \left(\sum_{n=0}^{\infty}p_d(n)q^n\right)\pmod2\notag\\
&\label{m2w1}\equiv q\dfrac{d}{dq}\left(\dfrac{f_2}{f_1}\right)
\pmod{2}.
\end{align}
Also,
\begin{equation}\label{dw1}
\dfrac{f_2}{f_1}=(-q;q)_\infty
 =\prod_{m=1}^{\infty}(1+q^m)\equiv\prod_{m=1}^{\infty}(1-q^m)
 =f_1
\pmod{2}.
\end{equation}
Differentiating \eqref{dw1} and then employing \eqref{m2w1},
we obtain
\begin{align}
\sum_{n=0}^{\infty}DSOME(n)q^n\equiv q\dfrac{d}{dq}f_1
&\label{D4}\equiv-\phi(-q^2)q\dfrac{d}{dq}f_1
 \equiv-q\dfrac{d}{dq}f_1
 \equiv q\dfrac{d}{dq}f_1
\pmod{2}.
\end{align}
From \eqref{DSE19}, we note that 
\begin{equation}\label{DSM2M6}
q\dfrac{d}{dq} f_1
=
q\dfrac{d}{dq}
\left(
\sum_{k\in\mathbb Z}
(-1)^kq^{k(3k-1)/2}
\right)
=
\sum_{k\in\mathbb Z}
(-1)^k
\dfrac{k(3k-1)}{2}
q^{k(3k-1)/2}.
\end{equation}
Employing \eqref{DSM2M6} in \eqref{D4}, we obtain
\begin{equation}\label{DSM2M7}
\sum_{n=0}^{\infty}DSOME(n)q^n
\equiv
\sum_{k\in\mathbb Z}
\dfrac{k(3k-1)}{2}
q^{k(3k-1)/2}
\pmod2.
\end{equation}
It is easily seen that, 
\begin{equation}\label{DSM2M10}
\dfrac{k(3k-1)}{2}\equiv1\pmod2
\text{ if and only if }
k\equiv1,2\pmod4.
\end{equation}
If $k\equiv1\pmod4$, then $k=4j+1$ for a unique
$j\in\mathbb Z$, and
\begin{equation}\label{DSM2M11}
\dfrac{k(3k-1)}{2}
=
\dfrac{(4j+1)(3(4j+1)-1)}{2}
=
24j^2+10j+1.
\end{equation}
Again, if $k\equiv2\pmod4$, then $k=4j+2$ for  unique
$j\in\mathbb Z$, and
\begin{equation}\label{DSM2M12}
\dfrac{k(3k-1)}{2}
=
\dfrac{(4j+2)(3(4j+2)-1)}{2}
=
24j^2+22j+5.
\end{equation}
Using   \eqref{DSM2M10}, \eqref{DSM2M11} and \eqref{DSM2M12} in \eqref{DSM2M7}, we obtain
\begin{align}
\sum_{n=0}^{\infty}DSOME(n)q^n
&\equiv
\sum_{\substack{k\in\mathbb Z\\k\equiv1\;(\mathrm{mod}\;4)}}
q^{k(3k-1)/2}
+
\sum_{\substack{k\in\mathbb Z\\k\equiv2\;(\mathrm{mod}\;4)}}
q^{k(3k-1)/2}
\pmod2
\notag\\
&\label{DSM2M15}\equiv
\sum_{j\in\mathbb Z}
q^{24j^2+10j+1}
+
\sum_{j\in\mathbb Z}
q^{24j^2+22j+5}
\pmod2.
\end{align}
Now equating the  coefficients of
$q^n$ in \eqref{DSM2M15}, we complete the proof.
\end{proof}

The proof of Corollary \ref{crm2} is similar to the proof of Corollary \ref{cor:DSOME-complete-family}, hence omitted.

\begin{corollary}\label{crm2}
Let $M$ and $r$ be integers such that $M\geq 1$ and $0\leq r<M$. Define
$$\mathcal{P}_{M}
=
\left\{
24x^{2}+10x+1\pmod{M}:x\in\mathbb{Z}
\right\}
\cup
\left\{
24x^{2}+22x+5\pmod{M}:x\in\mathbb{Z}
\right\}.$$
Then, we have
\begin{enumerate}
\item[$(i)$]
$DSOME(Mn+r)\equiv0\pmod2
\text{ holds for all integer $n\geq0$ if and only if }
r\notin\mathcal{P}_{M}.$
\item[$(ii)$] If $\gcd(M,6)=1$, then 
$DSOME(Mn+r)\equiv0\pmod2$
\text{ holds for all integer} $n\geq0$, \text{ if and only if }
$24r+1$ is not congruent modulo $M$ to a perfect square.
\item[$(iii)$] If $p\geq5$ is a prime, then
$$DSOME(pn+r)\equiv0\pmod2
\text{ holds for all integer $n\geq0$, if and only if }
\left(\dfrac{24r+1}{p}\right)=-1, $$
where $\left(\dfrac{\cdot}{p}\right)$ denotes the Legendre symbol.
\end{enumerate}
\end{corollary}

\begin{remark}
From Corollary \ref{crm2}, for every integer $n\geq0$, one can easily obtain  congruences  modulo 2 for
$DSOME(n)$ by considering different values of $M$. For example, setting $M=2$, $5$, $7$ and $11$, we have the following congruences, respectively:
\begin{align*}
DSOME(2n)
&\equiv0\pmod2,\\
DSOME(5n+t)
&
\equiv0\pmod2,\quad t=3, 4, \\
DSOME(7n+t)
&\equiv0\pmod2, \quad t=3, 4, 6\\\intertext{and}
DSOME(11n+t)&
\equiv0\pmod2, \quad t=3, 6, 8, 9, 10.
\end{align*}
\end{remark}

Proof of Corollary \ref{cio} is similar to the proof of Corollary \ref{cor:DSOME-odd-valuation}, hence omitted.

\begin{corollary}\label{cio}
Let $p\geq5$ be a prime and  $s$ be any nonnegative integer. If $B$ is a positive
integer such that $p\nmid B$ and
$
p^{2s+1}B\equiv1\pmod{24},
$
then
\begin{equation*}
DSOME
\left(
\dfrac{p^{2s+1}B-1}{24}
\right)
\equiv0\pmod2.
\end{equation*}
\end{corollary}

\begin{theorem}\label{T1}
For any integer $n\ge0$, we have
\begin{equation*}
DSOME(2n)\equiv
\begin{cases}
2\pmod4,
& n=48k^2+14k+1 \text{ or \;$48k^2+46k+11$}\text{ for some }k\in\mathbb Z,\\
0\pmod4,
& \text{otherwise}.
\end{cases}
\end{equation*}
\end{theorem}

\begin{proof}
Employing \eqref{DSE3} in \eqref{SDI10},  we obtain
\begin{equation}\label{DSE2}
\sum_{n=0}^{\infty}DSOME(n)q^n
=
\dfrac{f_1}{8}
\left(
\phi(-q)^{-1}-\phi(-q)^3
\right).
\end{equation}
From \eqref{xq}, we note that
\begin{align}
\label{DSE4}(1+2X(q))^{-1}
&\equiv
1-2X(q)+4X(q)^2-8X(q)^3+16X(q)^4
\pmod{32},
\\
\label{DSE4a}(1+2X(q))^3
&=
1+6X(q)+12X(q)^2+8X(q)^3,
\end{align}
 and
\begin{equation}\label{DSE5}
X(q)^2=
\sum_{m=1}^{\infty}q^{2m^2}
+
2\sum_{1\leq m<\ell}
(-1)^{m+\ell}q^{m^2+\ell^2}
\equiv
\sum_{m=1}^{\infty}q^{2m^2}
\equiv
\sum_{m=1}^{\infty}(-1)^m q^{2m^2}
=X(q^2)
\pmod2.
\end{equation}
Using  \eqref{DSE3}, \eqref{DSE4}, \eqref{DSE4a} and \eqref{DSE5} in \eqref{DSE2}, we obtain
\begin{align}
\sum_{n=0}^{\infty}DSOME(n)q^n
&\label{DSMN1}=
\dfrac{f_1}{8}
\left(
(1+2X(q))^{-1}-(1+2X(q))^3
\right)\\
&\equiv
f_1\left(-X(q)-X(q)^2-2X(q)^3+2X(q)^4\right)
\pmod4
\notag\\
&\label{DSE6}\equiv
-f_1X(q)(1+X(q))(1-2X(q)^2)
\pmod4.
\end{align}
From \eqref{DSE3} and \eqref{DSE5}, we obtain
\begin{align}
X(q)(1+X(q))
&\label{DSMN2}=
\dfrac{\phi(-q)^2-1}{4}\\\intertext{and}
1-2X(q)^2
&\label{DSE7}\equiv
1-2X(q^2)
\equiv
1+2X(q^2)
\equiv
\phi(-q^2)
\pmod4.
\end{align}
Employing \eqref{DSE3}, \eqref{DSMN2}, and \eqref{DSE7} in \eqref{DSE6}, we obtain
\begin{align}
\sum_{n=0}^{\infty}DSOME(n)q^n
&\equiv
\dfrac{f_1}{4}
\left(1-\phi(-q)^2\right)
\phi(-q^2)
\pmod4
\notag\\
&\equiv
\dfrac{f_1}{4}
\left(
1-\dfrac{f_1^4}{f_2^2}
\right)
\dfrac{f_2^2}{f_4}\pmod4
\notag\\
&\label{DSE8}\equiv
\dfrac{f_1(f_2^2-f_1^4)}{4f_4}
\pmod4.
\end{align}
Now, for every positive integer $m$, we note that
\begin{equation}\label{DSE9}
(1-q^m)^4=
(1-q^{2m})^2-4q^m(1-q^m)^2
=
(1-q^{2m})^2
\left(
1-\dfrac{4q^m}{(1+q^m)^2}
\right).
\end{equation}
Therefore, \eqref{DSE9} implies
\begin{align}
f_1^4
&=
\prod_{m=1}^{\infty}(1-q^m)^4
\notag\\
&=
\prod_{m=1}^{\infty}(1-q^{2m})^2
\prod_{m=1}^{\infty}
\left(
1-\dfrac{4q^m}{(1+q^m)^2}
\right)
\notag\\
&=
f_2^2
\prod_{m=1}^{\infty}
\left(
1-\dfrac{4q^m}{(1+q^m)^2}
\right)
\notag\\
&\label{DSE10}\equiv
f_2^2
\left(
1-4\sum_{m=1}^{\infty}
\dfrac{q^m}{(1+q^m)^2}
\right)
\pmod{16}
\end{align}
and 
\begin{equation}\label{DSE11}
\dfrac{f_2^2-f_1^4}{4}
\equiv
f_2^2
\sum_{m=1}^{\infty}
\dfrac{q^m}{(1+q^m)^2}
\pmod4.
\end{equation}
Employing \eqref{DSMN3}, \eqref{DSE13} and \eqref{DSE11}
 in \eqref{DSE8}, we obtain
\begin{equation}\label{DSE15}
\sum_{n=0}^{\infty}DSOME(n)q^n
\equiv
\dfrac{f_1f_2^2}{f_4}
\sum_{n=1}^{\infty}\sigma(n)q^n
\equiv
-\dfrac{f_2^2}{f_4}q\dfrac{d}{dq} f_1\pmod4
\equiv
-\phi(-q^2)q\dfrac{d}{dq} f_1
\pmod4.
\end{equation}
Employing \eqref{DSE16}, \eqref{DSE17} and \eqref{DSE3} 
 in \eqref{DSE15}, we obtain
\begin{align}
\sum_{n=0}^{\infty}DSOME(2n)q^n
&\equiv
-U_2\left[
\phi(-q^2)q\dfrac{d}{dq}f_1
\right]
\pmod4
\notag\\
&\equiv
-\phi(-q)U_2\left[q\dfrac{d}{dq} f_1\right]\pmod4
\notag\\
&\equiv
-2\phi(-q)q\dfrac{d}{dq} U_2\left[f_1\right]\pmod4
\notag\\
&\equiv
-2(1+2X(q))q\dfrac{d}{dq} U_2\left[f_1\right]\pmod4
\notag\\
&\label{DSE18}\equiv
2q\dfrac{d}{dq}U_2\left[f_1\right]
\pmod4.
\end{align}
Employing \eqref{DSE19}, we obtain
\begin{align}
U_2\left[f_1\right]
&=
\sum_{\substack{k\in\mathbb Z\\4\mid k(3k-1)}}
(-1)^kq^{k(3k-1)/4},
\notag\\\intertext{and}
2q\dfrac{d}{dq} U_2\left[f_1\right]
&\label{DSE20}=
2\sum_{\substack{k\in\mathbb Z\\4\mid k(3k-1)}}
\dfrac{k(3k-1)}{4}
(-1)^kq^{k(3k-1)/4}.
\end{align}
Now the term on the right-hand side of \eqref{DSE20} is nonzero modulo $4$
if and only if 
$k(3k-1)\equiv4\pmod8.$
Also, 
$$\begin{array}{c|cccccccc}
k\pmod8&0&1&2&3&4&5&6&7\\ \hline
k(3k-1)\pmod8&0&2&2&0&4&6&6&4
\end{array}.$$
So  $k(3k-1)\equiv4\pmod8$ if  $k\equiv4,7\pmod8$, and this implies that
\begin{align}\label{er}
\dfrac{(8j+4)(3(8j+4)-1)}{4}
&=
48j^2+46j+11\\\intertext{and}
\dfrac{(8j+7)(3(8j+7)-1)}{4}
&\label{DSE24}=
48j^2+82j+35
=
48(-j-1)^2+14(-j-1)+1.
\end{align}
Hence, $2(-1)^k\dfrac{k(3k-1)}{4}\equiv 2\pmod 4$ for $k\equiv4,7\pmod8$.
Using \eqref{er} and \eqref{DSE24} in \eqref{DSE18} and  then replacing $-j-1$
by $j$ in \eqref{DSE24}, we obtain
\begin{align}
\sum_{n=0}^{\infty}DSOME(2n)q^n
&\label{DSE25}\equiv
2\sum_{j\in\mathbb Z}
\left(
q^{48j^2+14j+1}
+
q^{48j^2+46j+11}
\right)
\pmod4.
\end{align}
Note that for any $k, l\in\mathbb Z$ and $k\ne l$, $48k^2+14k+1\ne 48l^2+14l+1$ and $48k^2+46k+11\ne 48l^2+46l+11$.  Also,   $48k^2+14k+1\ne 48l^2+46l+11$ for all $k,l\in\mathbb Z$.  Equating  the coefficients of $q^n$ in  \eqref{DSE25}, we complete the proof. 
\end{proof}

\begin{corollary}\label{cor:DSOME-complete-family}
Let $M$ and $r$ be integers such that $M\geq 1$ and
$0\leq r<M$. Define
\begin{equation*}
\begin{aligned}
\mathcal{Q}_{M}
=
\{48x^{2}+14x+1\pmod{M}:x\in\mathbb{Z}\}\cup
\{48x^{2}+46x+11\pmod{M}:x\in\mathbb{Z}\}.
\end{aligned}
\end{equation*}
Then, we have
\begin{enumerate} \item[$(i)$] $DSOME(2Mn+2r)\equiv0\pmod{4} \text{  holds for all integers $n\ge0$ if and only if } r\notin\mathcal{Q}_{M}.$
    \item[$(ii)$]If $\gcd(M,6)=1$, then 
$
DSOME(2Mn+2r)\equiv0\pmod{4}$
\text{  holds for all integers $n\ge0$} \text{ if and only if } 
$48r+1$ is not congruent modulo $M$ to a perfect square.

     \item[$(iii)$]If $p\geq5$ is a prime,  then 
$$
DSOME(2pn+2r)\equiv0\pmod{4}
\text{  holds for all integers $n\ge0$ if and only if }
\left(\dfrac{48r+1}{p}\right)=-1,
$$
where, here and throughout the paper, $\left(\dfrac{\cdot}{p}\right)$ denotes the Legendre symbol.
\end{enumerate}
\end{corollary}
\begin{proof}(i)~
Suppose that $r\notin\mathcal{Q}_{M}$ and for some integer
$n\geq0$,
\begin{equation*}
DSOME(2Mn+2r)\not\equiv0\pmod{4}.
\end{equation*}
Then by  Theorem \ref{T1}, for some integer $x$, either
$
Mn+r=48x^{2}+14x+1$
or 
$Mn+r=48x^{2}+46x+11,
$
which implies $r\in\mathcal{Q}_{M}$, a contradiction.

Conversely, suppose that
\begin{equation*}
DSOME(2Mn+2r)\equiv0\pmod{4}
\end{equation*}
holds for every integer $n\geq0$ and 
$r\in\mathcal{Q}_{M}$. Then, for some integer $x$, either
$
A(x)\equiv r\pmod{M}
$ or
$B(x)\equiv r\pmod{M},
$
where $A(x)$ and $B(x)$ are positive integers in $\mathcal{Q}_{M}$ such that
$A(x)=MN+r$ or $B(x)=MN'+r$, for some
nonnegative integer $N$ and $N'$. Therefore,  
Theorem \ref{T1} implies
$$
DSOME(2MN+2r)\equiv2\pmod{4}, \text{ or } DSOME(2MN'+2r)\equiv2\pmod{4},
$$
which is a contradiction. So the proof is complete.

(ii)~Let $\gcd(M,6)=1$. If $r\in\mathcal{Q}_{M}$, then clearly $48r+1$ is congruent modulo
$M$ to a perfect square.  

Conversely, suppose that for some integer $y$, 
$
y^{2}\equiv48r+1\pmod{M}.
$
Since $\gcd(M,6)=1$ implies $\gcd(M,48)=1$, 48 is invertible modulo $M$. So there exists an  integer  $x$ such
that 
$
48x+7\equiv y\pmod{M}.
$
Therefore, 
$$48\left(48x^{2}+14x+1\right)+1
=(48x+7)^{2}
\equiv y^{2}
\equiv48r+1
\pmod{M},$$
which implies  $r\in\mathcal{Q}_{M}$. This completes the proof of 
(ii).

(iii)~Let $p\geq5$ be a prime. Since $p\neq2,3$,  we have $\gcd(p,6)=1.$
Therefore, by (ii)
\begin{equation}\label{DSQ2}
r\in \mathcal{Q}_{p}
\quad\text{if and only if}\quad
48r+1 \text{ is a square modulo }p.
\end{equation}
By the definition of the Legendre symbol, we have
\begin{equation}\label{DSQ3}
\left(\dfrac{48r+1}{p}\right)
=
\begin{cases}
1, & 48r+1\not\equiv0\pmod p
\text{ and }48r+1\text{ is a perfect square modulo }p,\\[2mm]
0, & 48r+1\equiv0\pmod p,\\[2mm]
-1, & 48r+1\text{ is a quadratic nonresidue modulo }p.
\end{cases}
\end{equation}
Since $0$ is also a square modulo $p$, from \eqref{DSQ2} and
\eqref{DSQ3} it follows that 
$$r\notin \mathcal Q_p\text{ if and only if }
\left(\dfrac{48r+1}{p}\right)=-1.$$
Hence, the proof of (iii) is complete.
\end{proof}
\begin{remark}\label{rem:DSOME-examples}
From Corollary~\ref{cor:DSOME-complete-family}, one can easily obtain  congruences  modulo 4 for
$DSOME(n)$ by considering different values of $M$. For example, setting $M=2$, $5$, $7$ and $11$, we have the following congruences, respectively:
\begin{align*}
DSOME(4n)
&\equiv0\pmod{4},\\
DSOME(10n+t)
&\equiv0\pmod{4}; \quad t=4, 8,\\
DSOME(14n+t)
&\equiv0\pmod{4}; \quad t=4,6, 10,\\\intertext{and}
DSOME(22n+t)&\equiv0\pmod{4};\quad t=6, 8, 10, 14, 20.
\end{align*}
\end{remark}

\begin{corollary}\label{cor:DSOME-odd-valuation}
Let $p\geq5$ be a prime and  $s$ be any nonnegative integer. If $B$ is a positive integer such that $p\nmid B$ and $p^{2s+1}B\equiv1\pmod{48}$, then
\begin{equation*}
DSOME
\left(
\frac{p^{2s+1}B-1}{24}
\right)
\equiv0\pmod{4}.
\end{equation*}
\end{corollary}

\begin{proof}
Set
\begin{equation*}
N=\frac{p^{2s+1}B-1}{48}.
\end{equation*}
Then $N$ is a nonnegative integer and
$
48N+1=p^{2s+1}B.
$
Now if  $DSOME(2N)\not\equiv0\pmod{4}$ then 
Theorem  \ref{T1} implies that $48N+1$ is a perfect square, 
which is impossible as the exponent of $p$ in
$p^{2s+1}B$ is odd. Thus,
$DSOME(2N)\equiv0\pmod{4}$. 
\end{proof}

\begin{theorem}\label{DSGTHM}
For $k\in\mathbb Z$, let
\begin{equation*}
P_k=\dfrac{k(3k-1)}{2} 
\end{equation*}
 and for any integer $n\ge 0$, define
\begin{align*}
\rho(n)
&=
\#\left\{
(a,j)\in\mathbb Z_{>0}\times\mathbb Z:
n=2a^2+24j^2+10j+1
\right\}
\notag\\
&\quad+
\#\left\{
(a,j)\in\mathbb Z_{>0}\times\mathbb Z:
n=2a^2+24j^2+22j+5
\right\},
\end{align*}
where $\#S$ denotes the number of elements of the finite set $S$.
Then
$$DSOME(n)
\equiv
\begin{cases}
2\rho(n)-(-1)^kn\pmod4,
& n=P_k\text{ for some unique }k\in\mathbb Z,\\[2mm]
2\rho(n) \pmod4,
& n\neq P_k\text{ for every }k\in\mathbb Z.
\end{cases}$$
\end{theorem}
\begin{proof}
From \eqref{DSE3}, we have
\begin{equation}\label{DSG4}
\phi(-q^2)
=
1+2\sum_{a=1}^{\infty}(-1)^a q^{2a^2}
=
\sum_{a\in\mathbb Z}(-1)^a q^{2a^2}.
\end{equation}
From \eqref{DSE19}, we obtain
\begin{align}
q\dfrac{d}{dq}f_1
=
q\dfrac{d}{dq}
\left(
\sum_{k\in\mathbb Z}(-1)^kq^{P_k}
\right)
=
\sum_{k\in\mathbb Z}
(-1)^kP_kq^{P_k}.
\label{DSG5}
\end{align}
Employing  \eqref{DSG4} and \eqref{DSG5} in \eqref{DSE15}, we obtain
\begin{align}
\sum_{n=0}^{\infty}DSOME(n)q^n
&\equiv
-\left(
\sum_{a\in\mathbb Z}(-1)^a q^{2a^2}
\right)
\left(
\sum_{k\in\mathbb Z}(-1)^kP_kq^{P_k}
\right)
\pmod 4
\notag\\
&=
-\sum_{a\in\mathbb Z}
\sum_{k\in\mathbb Z}
(-1)^{a+k}P_kq^{2a^2+P_k}
\notag\\
&=
-\sum_{k\in\mathbb Z}
(-1)^kP_kq^{P_k}-
\sum_{a=1}^{\infty}
\sum_{k\in\mathbb Z}
\left(
(-1)^{a+k}+(-1)^{-a+k}
\right)
P_kq^{2a^2+P_k}
\notag\\
&\label{DSG6}=
-\sum_{k\in\mathbb Z}
(-1)^kP_kq^{P_k}-
2\sum_{a=1}^{\infty}
\sum_{k\in\mathbb Z}
(-1)^{a+k}P_kq^{2a^2+P_k}
\pmod 4.
\end{align}
For $a\geq1$ and $k\in\mathbb Z$, we note that
\begin{equation}\label{DSG7}
-2(-1)^{a+k}P_k
\equiv
\begin{cases}
0\pmod 4,
& P_k\text{ is even},\\
2\pmod 4,
& P_k\text{ is odd}.
\end{cases}
\end{equation}
Also, 
\begin{equation}\label{DSG10}
P_k\text{ is odd}
\text{ if and only if }
k\equiv1,2\pmod 4.
\end{equation}
Employing \eqref{DSG7} and \eqref{DSG10} in \eqref{DSG6}, we obtain
\begin{align}
\sum_{n=0}^{\infty}DSOME(n)q^n
&\equiv
-\sum_{k\in\mathbb Z}
(-1)^kP_kq^{P_k}
+
2\sum_{a=1}^{\infty}
\sum_{j\in\mathbb Z}
q^{2a^2+24j^2+10j+1}
\notag\\
&\label{DSG12}\quad+
2\sum_{a=1}^{\infty}
\sum_{j\in\mathbb Z}
q^{2a^2+24j^2+22j+5}
\pmod 4.
\end{align}
Now, comparing the coefficients of like powers of $q$ in \eqref{DSG12}, we arrive at the  desired result. 
\end{proof}
\begin{theorem}\label{DSM8THM}
 For any positive integer $n$, we have
\begin{equation*}
DSOME(n)
\equiv
3\sum_{r=0}^{n-1}
p_{\mathrm d}(r)
\left(
\sigma_3(n-r)+2\sigma(n-r)
\right)
\pmod 8.
\end{equation*}
\end{theorem}

\begin{proof}
Using \eqref{DSE2} and \eqref{DSE3}, we obtain
\begin{align}
\sum_{n=0}^{\infty}DSOME(n)q^n
=
\dfrac{f_1}{8}
\left(
\phi(-q)^{-1}-\phi(-q)^3
\right)
=
\dfrac{f_2}{f_1}\,
\dfrac{1}{8}
\left(
1-
\left(
\dfrac{f_1^4}{f_2^2}
\right)^2
\right).
\label{DSM8M3}
\end{align}
For every positive integer $m$, let
\begin{equation}\label{DSM8M4}
a_m=\dfrac{q^m}{(1+q^m)^2}.
\end{equation}
Then
$$\dfrac{f_1^4}{f_2^2}=
\dfrac{\displaystyle\prod_{m=1}^{\infty}(1-q^m)^4}
     {\displaystyle\prod_{m=1}^{\infty}(1-q^{2m})^2}=
\prod_{m=1}^{\infty}
\left(
1-\dfrac{4q^m}{(1+q^m)^2}
\right)=
\prod_{m=1}^{\infty}(1-4a_m).$$
Therefore,
\begin{equation}\label{DSM8M6}
\left(
\dfrac{f_1^4}{f_2^2}
\right)^2=
\prod_{m=1}^{\infty}(1-4a_m)^2=
\prod_{m=1}^{\infty}
\left(
1-8a_m+16a_m^2
\right).
\end{equation}
Also, for any  integer $M\geq2$, 
\begin{equation}\label{DSM8M7}
\prod_{m=1}^{M}
\left(
1-8a_m+16a_m^2
\right)
=
1+
\sum_{m=1}^{M}
\left(
-8a_m+16a_m^2
\right)\quad+
\sum_{t=2}^{M}
\sum_{1\leq m_1<\cdots<m_t\leq M}
\prod_{i=1}^{t}
\left(
-8a_{m_i}+16a_{m_i}^2
\right),
\end{equation} and for any integer 
 $t\geq2$, 
\begin{equation}\label{DSM8M8}
\prod_{i=1}^{t}
\left(
-8a_{m_i}+16a_{m_i}^2
\right)=
\prod_{i=1}^{t}
8\left(
-a_{m_i}+2a_{m_i}^2
\right)
=
8^t
\prod_{i=1}^{t}
\left(
-a_{m_i}+2a_{m_i}^2
\right)
\equiv0\pmod{64}.
\end{equation}
So employing \eqref{DSM8M8} in \eqref{DSM8M7}, we obtain
\begin{equation}\label{DSM8M9}
\prod_{m=1}^{M}
\left(
1-8a_m+16a_m^2
\right)
\equiv
1+
\sum_{m=1}^{M}
\left(
-8a_m+16a_m^2
\right)
=
1-8\sum_{m=1}^{M}a_m
+16\sum_{m=1}^{M}a_m^2
\pmod{64},
\end{equation}
Now, taking limit as
$M\rightarrow\infty$ in \eqref{DSM8M9} and employing
\eqref{DSM8M6}, we obtain
\begin{equation*}
1-
\left(
\dfrac{f_1^4}{f_2^2}
\right)^2
\equiv
8\sum_{m=1}^{\infty}a_m
-16\sum_{m=1}^{\infty}a_m^2
\pmod{64},
\end{equation*}
which implies
\begin{equation}\label{DSM8M11}
\dfrac{1}{8}
\left(
1-
\left(
\dfrac{f_1^4}{f_2^2}
\right)^2
\right)
\equiv
\sum_{m=1}^{\infty}a_m
-
2\sum_{m=1}^{\infty}a_m^2
\pmod8.
\end{equation}
From  \eqref{DSM8M4}, we note that 
\begin{align}
a_m-2a_m^2
=
\dfrac{q^m}{(1+q^m)^2}
-
\dfrac{2q^{2m}}{(1+q^m)^4}
=
\dfrac{q^m(1+q^{2m})}{(1+q^m)^4}.
\label{DSM8M12}
\end{align}
Employing  \eqref{DSM8M11} and
\eqref{DSM8M12} in \eqref{DSM8M3}, we obtain
\begin{align}
\sum_{n=0}^{\infty}DSOME(n)q^n
&\equiv
\dfrac{f_2}{f_1}
\left(
\sum_{m=1}^{\infty}a_m
-
2\sum_{m=1}^{\infty}a_m^2
\right)
\equiv
\dfrac{f_2}{f_1}
\sum_{m=1}^{\infty}
\dfrac{q^m(1+q^{2m})}{(1+q^m)^4}
\pmod8.
\label{DSM8M13}
\end{align}
Substituting $x=q^m$ in \eqref{DSM8M17}, we obtain
\begin{align}
\sum_{m=1}^{\infty}
\dfrac{q^m(1+q^{2m})}{(1+q^m)^4}
&=
\dfrac{1}{3}
\sum_{m=1}^{\infty}
\sum_{j=1}^{\infty}
(-1)^{j-1}j(j^2+2)q^{mj}
\notag\\
&=
\dfrac{1}{3}
\sum_{N=1}^{\infty}
\left(
\sum_{j\mid N}
(-1)^{j-1}j(j^2+2)
\right)q^N
\notag\\
&\label{DSM8M19}=
\dfrac{1}{3}
\sum_{N=1}^{\infty}
\left(
\sum_{j\mid N}(-1)^{j-1}j^3
+
2\sum_{j\mid N}(-1)^{j-1}j
\right)q^N.
\end{align}
Taking $s=3$ and $s=1$ in \eqref{DSM8M20}, we obtain
\begin{align}
\sum_{j\mid N}(-1)^{j-1}j^3
=
\sigma_3(N)-16\sigma_3(N/2)\quad 
\text{and}\quad 
\sum_{j\mid N}(-1)^{j-1}j
=
\sigma(N)-4\sigma(N/2).
\label{DSM8M21}
\end{align}
Employing \eqref{DSM8M21} in \eqref{DSM8M19}, we obtain
\begin{align}
3
\left(
\sum_{m=1}^{\infty}
\dfrac{q^m(1+q^{2m})}{(1+q^m)^4}
\right)
&=
\sum_{N=1}^{\infty}\left(\sigma_3(N)-16\sigma_3(N/2)+
2\left(
\sigma(N)-4\sigma(N/2)
\right)\right)q^N
\notag\\
&\equiv
\sum_{N=1}^{\infty}\left(\sigma_3(N)+2\sigma(N)\right)q^N
\pmod8,\notag
\end{align}
which gives
\begin{equation}\label{DSM8M24}
\sum_{m=1}^{\infty}
\dfrac{q^m(1+q^{2m})}{(1+q^m)^4}
\equiv
3\sum_{N=1}^{\infty}
\left(
\sigma_3(N)+2\sigma(N)
\right)q^N
\pmod8.
\end{equation}
Employing  \eqref{DSM8M25} and \eqref{DSM8M24} in \eqref{DSM8M13}, we obtain
\begin{align}
\sum_{n=0}^{\infty}DSOME(n)q^n
&\equiv
3
\left(
\sum_{r=0}^{\infty}p_{\mathrm d}(r)q^r
\right)
\left(
\sum_{N=1}^{\infty}\sigma_3(N)q^N
\right)+
6
\left(
\sum_{r=0}^{\infty}p_{\mathrm d}(r)q^r
\right)
\left(
\sum_{N=1}^{\infty}\sigma(N)q^N
\right)
\pmod8
\notag\\
&\equiv
3\sum_{n=1}^{\infty}
\left(
\sum_{r=0}^{n-1}
p_{\mathrm d}(r)\sigma_3(n-r)
\right)q^n+
6\sum_{n=1}^{\infty}
\left(
\sum_{r=0}^{n-1}
p_{\mathrm d}(r)\sigma(n-r)
\right)q^n
\pmod8
\notag\\
&\label{DSM8M26}\equiv
3\sum_{n=1}^{\infty}
\left(
\sum_{r=0}^{n-1}
p_{\mathrm d}(r)
\left(
\sigma_3(n-r)+2\sigma(n-r)
\right)
\right)q^n
\pmod8.
\end{align}
Comparing the coefficients of $q^n$ in \eqref{DSM8M26}, we complete the proof.
\end{proof}

\begin{theorem}\label{nd1}
For every nonnegative integer $n$ and $1\leq t \leq 4$, we have
$$DSOME(50n+10t+1)\equiv 0\pmod{8}.$$
\end{theorem}

\begin{proof}
From \eqref{DSMN1}, we obtain
\begin{equation}\label{D8M7}
\sum_{n=0}^{\infty}DSOME(n)q^n\equiv
f_1\bigl(
-X(q)-X(q)^2-2X(q)^3
+2X(q)^4+4X(q)^5
\bigr)
\pmod{8}.
\end{equation}
Let
$$Y(q)=\sum_{n=0}^{\infty}y(n)q^n=f_1X(q)(1+X(q)),$$
and
\begin{equation}\label{D8M9}
R(q)=\sum_{n=0}^{\infty}r(n)q^n=f_1\left(X(q)^4-X(q)^3+2X(q)^5\right).
\end{equation}
Then \eqref{D8M7} becomes
$$\sum_{n=0}^{\infty}DSOME(n)q^n
\equiv -Y(q)+2R(q)\pmod{8}.$$
From \eqref{DSE3} and \eqref{DSMN2}, we obtain
\begin{equation}\label{D8M22}
X(q)(1+X(q))
=\dfrac{1}{4}
\left(
\dfrac{f_1^4}{f_2^2}-1
\right).
\end{equation}
From  \cite[(10.7.3)]{md}, we note that
\begin{equation}\label{D8M12}
\dfrac{f_1^5}{f_2^2}
=
\sum_{j\in\mathbb Z}(1-6j)
q^{j(3j-1)/2}.
\end{equation}
Multiplying \eqref{D8M22} by $f_1$ and employing
\eqref{DSE19} and \eqref{D8M12}, we obtain
\begin{equation*}
Y(q)
=\dfrac{1}{4}
\left(
\dfrac{f_1^5}{f_2^2}-f_1
\right)=
\sum_{j\in\mathbb Z}
\left(\dfrac{1-6j-(-1)^j}{4}\right)
q^{j(3j-1)/2}.
\end{equation*}
Fix $1\leq t\leq4$ and set
\begin{equation}\label{D8M26}
N=50n+10t+1
\quad\mbox{and}\quad
M=24N+1.
\end{equation}
Then
\begin{equation}\label{D8M27}
M
=5(240n+48t+5).
\end{equation}
Since
$240n+48t+5\equiv3t\not\equiv0\pmod{5},
\qquad 1\leq t\leq4,$
\begin{equation}\label{D8M29}
\nu_5(M)=1,
\end{equation}
where $\nu_5(M)$ denotes the $5$-adic valuation of $M$ (that is, $\nu_5(M)$ is the highest exponent of 5 in the prime factorisation of  $M$).

Also, if $N=j(3j-1)/2$ for some $j\in\mathbb Z$, then
\begin{equation*}
M
=(6j-1)^2, 
\end{equation*}
 which implies $\nu_5(M)>1$, which is a contradiction. Therefore, coefficient of $q^N$ in the series expansion of $Y(q)$ is $0$.
 
Again, employing  \eqref{DSMN3},  \eqref{DSE13} and \eqref{DSE10}
 in \eqref{D8M22}, we obtain
\begin{equation}\label{D8M40}
X(q)(1+X(q))
=
\dfrac{1}{4}
\left(
\dfrac{f_1^4}{f_2^2}-1
\right)\equiv
-\sum_{\ell=1}^{\infty}\sigma(\ell)q^\ell
\pmod{4}.
\end{equation}
Employing \eqref{DSMN3} in \eqref{D8M40}, we obtain
\begin{equation}\label{D8A1}
\dfrac{q\dfrac{d}{dq}f_1}{f_1}
\equiv X(q)+X(q)^2
\pmod{4}.
\end{equation}
Logarithmically differentiating \eqref{DSE3}, we obtain
$$\dfrac{2q\dfrac{d}{dq}X(q)}{1+2X(q)}
=
2\dfrac{q\dfrac{d}{dq}f_1}{f_1}
-\dfrac{q\dfrac{d}{dq} f_2}{f_2}, $$
which on employing \eqref{D8A1} gives
\begin{align}
 q\dfrac{d}{dq}   X(q)
&=
\left(1+2X(q)\right)
\left(
\dfrac{q\dfrac{d}{dq}f_1}{f_1}
-
\left.
\dfrac{ q\dfrac{d}{dq}f_1}{f_1}
\right|_{q\mapsto q^2}
\right) \notag\\
&\label{D8A2}\equiv
\left(1+2X(q)\right)
\left(
X(q)+X(q)^2-X(q^2)-X(q^2)^2
\right)
\pmod{4}.
\end{align}
Furthermore,
\begin{equation}\label{DSMN4}
q\dfrac{d}{dq}\left(f_1X(q)\right)
=
f_1\left(
X(q)\dfrac{ q\dfrac{d}{dq}f_1}{f_1}
+  q\dfrac{d}{dq}X(q)
\right).    
\end{equation}
Combinign \eqref{D8A1}, \eqref{D8A2}, and \eqref{DSMN4}, we obtain
\begin{align}
	\dfrac{
		f_1X(q)-q\dfrac{d}{dq}\left(f_1X(q)\right)-f_1X(q^2)
	}{f_1}
	&=
	X(q)-X(q)\dfrac{q\dfrac{d}{dq} f_1}{f_1}
	-q\dfrac{d}{dq}X(q)-X(q^2) \notag\\
	&\label{D8A3}\equiv
	-4X(q)^2-3X(q)^3+X(q^2)^2
	+2X(q)X(q^2)\notag\\
	&\hspace{1cm}+2X(q)X(q^2)^2
	\pmod{4}.
\end{align}
Employing \eqref{DSE5} in \eqref{D8A3}, we obtain
\begin{align}
\dfrac{
f_1X(q)-q\dfrac{d}{dq}\left(f_1X(q)\right)-f_1X(q^2)
}{f_1}
&\equiv
-3X(q)^3+X(q)^4+2X(q)^3+2X(q)^5\notag\\
&\label{nw1}\equiv
X(q)^4-X(q)^3+2X(q)^5
\pmod{4}.
\end{align}
Next, set
\begin{equation}\label{nw2}
f_1X(q)=\sum_{\ell=0}^{\infty}a(\ell)q^\ell
\qquad\mbox{and}\qquad 
f_1X(q^2)=\sum_{\ell=0}^{\infty}b(\ell)q^\ell.
\end{equation}
Using \eqref{nw1} and \eqref{nw2} in \eqref{D8M9} and further comparing the coefficient of $q^N$ from both sides, we obtain
\begin{align*}
r(N)\equiv
(1-N)a(N)
-b(N)
\pmod{4}.
\end{align*}
We claim that $r(N)\equiv 0\pmod 4$, for it is sufficient to show that $b(N)=0$ and $a(N)\equiv 0\pmod 2$ as $(1-N)$ is even. 

Using \eqref{DSE19} and \eqref{xq}, we obtain
\begin{equation}\label{f1xq2}
f_1X(q^2)=\sum_{\ell=0}^{\infty}b(\ell)q^\ell=\sum_{j\in\mathbb Z}\sum_{m=1}^{\infty} 
(-1)^{j+m}
q^{j(3j-1)/2+2m^2}.
\end{equation}
Now, we show 
$
b(N)=0.
$
If for some $j\in\mathbb Z$ and integer $m\geq1$, let 
\begin{equation}\label{DSMN6}
N=\dfrac{j(3j-1)}{2}+2m^2.
\end{equation}
Employing \eqref{DSMN6} in \eqref{D8M26}, we obtain
\begin{equation}\label{D8A7}
M=24N+1=(6j-1)^2+48m^2.
\end{equation}
Employing \eqref{D8M27} in \eqref{D8A7}, we obtain
\begin{equation}\label{DSMN7}
(6j-1)^2-2m^2\equiv0\pmod{5}.
\end{equation}
Suppose that 5 does not divide $m$. Then $\gcd(5, m)=1$ and there exists a positive integer $m^\ast$ such that $(mm^\ast)^2\equiv 1\pmod 5$.  Therefore, \eqref{DSMN7} implies 
$$\left((6j-1)m^\ast\right)^2\equiv2\pmod{5},$$
which is impossible because $2$ is a quadratic nonresidue modulo $5$.
So 5 must divide both $m$ and 
$6j-1$.  Thus, by  \eqref{D8A7} and \eqref{DSMN7}, 25 divides $M$,  which is a contradiction to 
 \eqref{D8M29}. So $N\neq\dfrac{j(3j-1)}{2}+2m^2$ and  by \eqref{f1xq2} $b(N)=0$.\\
It remains to prove  that
\begin{equation*}
a(N)\equiv 0 \pmod{2}.
\end{equation*}
Using \eqref{DSE19} and \eqref{xq}, we obtain
\begin{equation}
\sum_{\ell=0}^{\infty}a(\ell)q^\ell
=f_1X(q)=
\sum_{j\in\mathbb Z}
\sum_{m=1}^{\infty}
(-1)^{j+m}
q^{j(3j-1)/2+m^2}.
\label{ANP2}
\end{equation}
Extracting the coefficient of $q^N$ from both sides of
\eqref{ANP2}, we obtain
\begin{equation*}
a(N)=
\sum_{\substack{j\in\mathbb Z,\;m\geq 1\\
N=j(3j-1)/2+m^2}}
(-1)^{j+m}, 
\end{equation*}
which implies
$$a(N)\equiv
\#\left\{
(j,m)\in\mathbb Z\times\mathbb Z_{\geq1}:
N=\dfrac{j(3j-1)}{2}+m^2
\right\}
\pmod{2}.$$
Therefore, it is enough to show that the number of solutions $(j, m)$ satisfying 
\begin{equation}\label{ANP6}
N=\dfrac{j(3j-1)}{2}+m^2,
\qquad
j\in\mathbb Z,\quad m\geq1,
\end{equation}
is always even.\\\\
For every solution $(j,m)$ of \eqref{ANP6}, let
\begin{equation}\label{ANP7}
u=6j-1,
\end{equation}
then by employing \eqref{ANP6} in \eqref{D8M26}, we obtain
\begin{equation*}
M
=24N+1 =(6j-1)^2+24m^2=u^2+24m^2\equiv u^2-m^2=(u-m)(u+m)\pmod 5.
\end{equation*}
Therefore,  either $5\mid u+m$ or $5\mid u-m$. Also, we note that 5 cannot divide both $u+m$ and $u-m$ simultaneously, otherwise  $\nu_5(M)\ge 2$, a contradiction.
Hence, there is a
unique $s\in\{1,-1\}$ such that 5 divides $u+sm$.
Define, 
\begin{equation}\label{ANP17}
u_1=\dfrac{24sm-u}{5},
\qquad
m_1=\dfrac{|su+m|}{5}.
\end{equation}
Clearly, $u_1\in\mathbb Z$ and $m_1\in\mathbb Z_{>0}$. Also, 
\begin{equation*}
u_1^2+24m_1^2=
\dfrac{(24sm-u)^2+24(su+m)^2}{25} =u^2+24m^2 =M.
\end{equation*}
Moreover, from \eqref{ANP7} and \eqref{ANP17}, 
$u_1\equiv-1\pmod{6}.$
Therefore,
$
j_1=\dfrac{u_1+1}{6}
$ is an integer. So, 
\begin{equation*}
N
=\dfrac{M-1}{24} =\dfrac{j_1(3j_1-1)}{2}+m_1^2.
\end{equation*}
Thus, for any solution $(j, m)$ of \eqref{ANP6} there always exists another  solution $(j_1,m_1)$. In fact, the solution $(j_1, m_1)$ corresponding to the solution $(j, m)$ is always unique. To see this,
define
$$s_1=
\begin{cases}
s,  & \text{if } su+m>0,\\
-s, & \text{if } su+m<0.
\end{cases}$$
Then, in the both cases,
$$s_1m_1=\dfrac{u+sm}{5},
\qquad
u_1+s_1m_1=5sm,
\qquad\mbox{and}\qquad 
|s_1u_1+m_1|=5m, $$
which implies
$$\dfrac{24s_1m_1-u_1}{5}
=\dfrac{24(u+sm)-(24sm-u)}{25}
=u \qquad \text{ and }\qquad \dfrac{|s_1u_1+m_1|}{5}
=m.$$\\\\
Thus, applying same construction as of $(u,m)$ to $(u_1,m_1)$ returns back $(u,m)$. Moreover, if $u_1=u$, then
$$\dfrac{24sm-u}{5}=u
\quad\Longrightarrow\quad
u=4sm,$$
which is impossible since by  \eqref{ANP7}, $u=6j-1$ is odd. Therefore, the solutions  of  \eqref{ANP6} occur in distinct pairs which implies $a(N)$ is even and so $a(N)\equiv 0\pmod 2$. Hence, the proof is complete by noting that $y(N)=0$ and $r(N)\equiv 0\pmod 4$.
\end{proof}

\begin{remark} The conjecture \eqref{SDI11} due to Baruah and Gogoi  is the particular case  $t=2$ of Theorem \ref{nd1}. \end{remark}

 \section*{\bf Declarations}
   
 \noindent\textbf{Funding}: This research did not receive funding.\\
 \noindent{\bf Author Contributions.} Both authors contributed equally to this work.
 
 \noindent{\bf Conflict of Interest.} The authors declare that there is no conflict of interest regarding the publication of this paper.
 
 \noindent{\bf Human and Animal Rights.} The authors declare that there is no research involving human participants or animals in the context of this paper.	
 
 \noindent{\bf Data Availability Statement.} Data sharing is not applicable to this paper as no datasets were generated or analyzed during the current study.	
\bibliographystyle{plain}

\end{document}